\documentclass[12pt]{amsart}   

\usepackage{amsmath, amssymb, amsthm}   

\usepackage{graphicx}          

\usepackage{hyperref}

\def\E{\mathbb{E}}
\def\P{\mathbb{P}}
\def\V{\mathbb{V}}

\def\Po{\mathcal{P}}

\newtheorem{thm}{Theorem}[section]
\newtheorem{prop}[thm]{Proposition}
\newtheorem{defin}[thm]{Definition}
\newtheorem{lemma}[thm]{Lemma}
\newtheorem{conj}[thm]{Conjecture}
\newtheorem{cor}[thm]{Corollary}

\begin{document}

\title{Profiles of permutations}
 
\author{Michael Lugo}
\address{Department of Mathematics, University of Pennsylvania, 209 South 33rd Street, Philadelphia, PA 19104}
\email{mlugo@math.upenn.edu}


\begin{abstract}
This paper develops an analogy between the cycle structure of, on the one hand, random permutations with cycle lengths restricted to lie in an infinite set $S$ with asymptotic density $\sigma$ and, on the other hand, permutations selected according to the Ewens distribution with parameter $\sigma$.  In particular we show that the asymptotic expected number of cycles of random permutations of $[n]$ with all cycles even, with all cycles odd, and chosen from the Ewens distribution with parameter $1/2$ are all ${1 \over 2} \log n + O(1)$, and the variance is of the same order.  Furthermore, we show that in permutations of $[n]$ chosen from the Ewens distribution with parameter $\sigma$, the probability of a random element being in a cycle longer than $\gamma n$ approaches $(1-\gamma)^\sigma$ for large $n$. The same limit law holds for permutations with cycles carrying multiplicative weights with average $\sigma$.  We draw parallels between the Ewens distribution and the asymptotic-density case and explain why these parallels should exist using permutations drawn from weighted Boltzmann distributions.  

\end{abstract}

\maketitle

\section{Introduction} 

 In this paper we study the cycle structure of random permutations in which the lengths of all cycles are constrained to lie in some infinite set $S$, and permutations may be made more or less likely to be chosen through multiplicative weights placed on their cycles.   Cycle structures viewed in this manner are a special case of certain measures on $S_n$ which are conjugation-invariant and assign a weight to each element of $S_n$ based on its cycle structure.
\begin{defin}\label{def:weighting-sequence} Let $\vec{\sigma} = (\sigma_1, \sigma_2, \ldots)$ be an infinite sequence of nonnegative real numbers.  Then the {\it weight} of the permutation $\pi \in S_n$, with respect to $\vec{\sigma}$, is
\[ w_{\vec{\sigma}} (\pi) = \prod_{i=1}^n \sigma_i^{c_i(\pi)} \]
where $c_i(\pi)$ is the number of cycles of length $i$ in $\pi$.  \end{defin}
Informally, each cycle in a permutation receives a weight depending on its length, and the weight of a permutation is the product of the weights of its cycles.  The sequence $\vec{\sigma}$ is called a {\it weighting sequence}.

For each positive integer $n$, let $(\Omega^{(n)}, \mathcal{F}^{(n)})$ be a probability space defined as follows.  Take $\Omega^{(n)} = S_n$, the set of permutations of $[n]$, and let $\mathcal{F}^{(n)}$ be the set of all subsets of $S_n$.  Endow $(\Omega^{(n)}, \mathcal{F}^{(n)})$ with a probability measure $\P_{\vec{\sigma}}^{(n)}$ for each weighting sequence $\vec{\sigma}$ as follows.  Let $\P_{\vec{\sigma}}^{(n)}(\pi) ={ {w_{\vec{\sigma}}(\pi)} / \sum_{\pi^\prime \in S_n} w_{\vec{\sigma}}(\pi^\prime)}$; that is, each permutation has probability proportional to its weight.  Extend $\P_{\vec{\sigma}}^{(n)}$ to all subsets of $S_n$ by additivity.   To streamline the notation, we will sometimes write $\P_{\vec{\sigma}}(\pi)$ for $\P_{\vec{\sigma}}^{(n)}(\pi)$.  The sum of the weights of $\vec{\sigma}$-weighted permutations of $[n]$ is
\[ \sum_{\pi \in S_n} w_{\vec{\sigma}}(\pi) = n! [z^n] \exp \left( \sum_{k \ge 1} \sigma_k z^k/k \right) \]
by the exponential formula for labelled combinatorial structures.

We fix some notation. Define the random variable $X_k^{(n)}: \Omega^{(n)} \to \mathbb{Z}^+$ by setting $X_k^{(n)}(\pi)$ equal to the number of $k$-cycles in the permutation $\pi$.  Let $X^{(n)}(\pi) = \sum_{k=1}^n X_k^{(n)}(\pi)$ be the total number of cycles.  We will often suppress $\pi$ and $(n)$ in the notation, and we will write (for example) $\P_{\vec{\sigma}}(X_1=1)$ as an abbreviation for $\P_{\vec{\sigma}} ( \{ \pi: X_1^{(n)}(\pi) = 1 \} )$.   Let $Y_k = kX_k$.  We define $Y_k$ in order to simplify the statement of some results.

This model incorporates various well-known classes of permutations, including generalized derangements (permutations in which a finite set of cycle lengths is prohibited), and the Ewens sampling formula from population genetics \cite{Ew72}, which corresponds to the weighting sequence $(\sigma, \sigma, \sigma, \ldots)$.  If $\vec{\sigma}$ is a 0-1 sequence with finitely many 1s, then this model specializes to random permutations of which all cycle lengths lie in a finite set.  These have a fascinating structure studied by Benaych-Georges \cite{B-G08} and Timashev \cite{Ti08}; a typical permutation of $[n]$ with cycle lengths in a finite set $S$ has about ${1 \over k} n^{k/\max S}$ $k$-cycles, for each $k$ in $S$.  In particular, most cycles are of length $\max S$, which may be unexpected at first glance.  Analytically, this situation is studied via the asymptotics of $[z^n] e^{P(z)}$ where $P$ is a polynomial, as done by Wilf \cite{Wi86}.  Yakymiv \cite{Ya00} has studied the case, alluded to by Bender \cite{Be74}, in which $\vec{\sigma}$ is a sequence of $0$s and $1$s with a fixed density $\sigma$ of $1$s; the behavior of such permutations is in broad outline similar to that of the Ewens sampling formula with parameter $\sigma$.  An ``enriched'' version of the model has been studied by Ueltschi and coauthors \cite{GRU07, UB08}. In their model, permutations are endowed with a spatial structure.  Each element of the ground set of the permutation is a point in the plane, and weights involve distances between points. Their ``simple model of random permutations with cycle weight'' \cite[Sec. 2]{UB08} is the model used here, where $\sigma_i = e^{-\alpha_i}$. 

There are other combinatorially interesting conjugation-invariant measures on $S_n$, including permutations with all cycle lengths distinct \cite{GK90}, and permutations with $k$th roots for some fixed $k$ \cite{FFGPP06, Po02}.  However the generating functions counting these classes are not exponentials of ``nice'' functions and thus different techniques are required.

Throughout this paper, we often implicitly assume that permutations under the uniform measure on $S_n$ are the ``primitive'' structure, and weighted permutations are a perturbation of these.  Here we follow Arratia {\it et al.} in \cite{ABT97, ABT03}, in embracing a similar philosophy and viewing the permutation as the archetype of a class of ``logarithmic combinatorial structures'', and Flajolet and Soria's definition of functions of logarithmic type \cite{FS90}.  

It will be convenient to use bivariate generating functions which count permutations by their size and number of cycles.  In general, we take $F(z,u) = \sum_{n,k} f_{n,k} {z^n \over n!} u^k$ to be the bivariate generating function, exponential in $z$ and ordinary in $u$, of a combinatorial class $\mathcal{F}$, where $f_{n,k}$ is the number of objects in $\mathcal{F}$ of size $n$ and with a certain parameter equal to $k$.   In our case $n$ will be the number of elements of a permutation, and $k$ the total number of cycles or the number of cycles of a specified size.  Then $[z^n] \left. {\partial \over \partial u} F(z,u) \right|_{u=1} / [z^n] F(z,1)$ gives the expected value of the parameter $k$ for an object of size $n$ selected uniformly at random.  The following lemma will frequently be useful, as it reduces the bivariate analysis to a univariate analysis.

\begin{lemma}\label{fundamental-lemma} Let $f(z)$ be the exponential generating function of permutations with weight sequence $\vec{\sigma}$.  Then the expected number of $k$-cycles in a permutation chosen according to the measure $\P_{\vec{\sigma}}^{(n)}$ is
\[ \E_{\vec{\sigma}}^{(n)} X_k = {\sigma_k \over k} {[z^{n-k}] f(z) \over [z^n] f(z)}.\] \end{lemma}

\begin{proof} The bivariate generating function counting the cycles of such permutations is  \[ \sigma_1 z + \sigma_2 {z^2 \over 2} + \cdots + \sigma_{k-1} {z^{k-1} \over k-1} + u \sigma_k {z^k \over k} + \sigma_{k+1} {z^{k+1} \over k+1} + \cdots \] and this can be rewritten as $(u-1) {\sigma_k z^k \over k} + \sum_{j \ge 1} {\sigma_j z^j \over j}$.  Thus, from the exponential formula, the bivariate generating function counting such permutations is 
\[ P(z,u) = \exp \left( (u-1) {\sigma_k z^k \over k} + \sum_{j \ge 1} {\sigma_j z^j \over j} \right). \]  The expected number of cycles in a random permutation is $[z^n] P_u(z,1)/[z^n] P(z,1)$, giving the result.
\end{proof}
The structure of this paper is as follows.  In Section 2 we give exact formulas and asymptotic series (Propositions \ref{ewens-exp-cycles} and \ref{ewens-var-cycles}) for the mean and variance of the number of cycles of permutations chosen from the Ewens distribution.  We also consider the average number of $k$-cycles in such permutations of $[n]$ for fixed $k$ (Propositions \ref{prop-ewens-exp-k-cycles} and \ref{ewens-asympt-k-cycles}) and for $k = \alpha n$ (Proposition \ref{ewens-asympt-alpha-n-cycles}).  An ``integrated'' version of these results, Theorem \ref{bulk-profile-general}, is one of the main results; this is a limit law for the probability that a random element of a weighted permutation is in a cycle within a certain prescribed range of lengths.   In Section 3 we derive similar results for permutations in which all cycle lengths have the same parity.  In addition, we determine the mean and variance of the number of cycles of  such permutations (Theorem \ref{thm-mean-cycles-odd} treats the odd case, and Theorem \ref{thm-mean-cycles-even} treats the even case).  In Section 4 we explore connections to the generation of random objects by Boltzmann sampling. The main theorem of this section, Theorem \ref{thm-boltzmann-concentrated}, states that the Boltzmann-sampled permutations of a certain class of approximate size $n$, including the Ewens and parity-constrained cases, have their number of cycles distributed with mean and variance approximately a constant multiple of $\log n$.  

\section{The Ewens sampling formula and Bernoulli decomposition} 

The {\it Ewens distribution} \cite{Ew72} on permutations of $[n]$ with parameter $\sigma$ gives to each permutation $\pi$ probability proportional to $\sigma^{X(\pi)}$.  This corresponds to the weighting sequence $\vec{\sigma} = (\sigma, \sigma, \sigma, \ldots)$; we will write $\P_\sigma^{(n)}, \E_\sigma^{(n)}$ for $\P_{\vec{\sigma}}^{(n)}, \E_{\vec{\sigma}}^{(n)}$, and call a random permutation selected in this manner a {\it $\sigma$-weighted permutation}.  In this section we derive formulas for the mean and variance of the number of cycles of permutations chosen from the Ewens distribution. Note that the number of cycles can be decomposed into a sum of independent Bernoulli random variables.  Similar decompositions are due to Arratia et al. in \cite[Sec. 5.2]{ABT03} for general $\sigma$, and Feller \cite[(46)]{Fe45} for $\sigma = 1$; the fact that the number of cycles is normally distributed is seen in \cite[Example 1]{FS90}.  Thus this section is largely expository; the proofs are provided for the purpose of comparison with other proofs to be given below.  The asymptotic series for $\E_\sigma^{(n)}$ and $\V_\sigma^{(n)}$ appear to be new.

\begin{thm}\label{ewens-bernoulli}\cite[Exercise 3.2.3]{Pi06} The distribution of the random variable $X$ under the measure $\P_{\sigma}^{(n)}$ is that of the sum $\sum_{k=1}^n Z_k$, where the $Z_k$ are independent random variables and $Z_k$ has the Bernoulli distribution with mean $\sigma/(\sigma+k-1)$. \end{thm}

\begin{proof}  The generating function of permutations of $[n]$ counted by their number of cycles is $\sum_{k=1}^n S(n,k) u^k = u(u+1)(u+2) \cdots (u+n-1)$, where $S(n,k)$ are the Stirling cycle numbers.   Replacing $u$ with $\sigma u$ and normalizing gives the probability generating function for the number of cycles,
\[ \sum_{k=1}^n S(n,k) \sigma^k u^k = {\sigma u \over \sigma} {\sigma u + 1 \over \sigma + 1} \cdots {\sigma u + n - 1 \over \sigma + n -1}, \]
and each factor is the probability generating function for a Bernoulli random variable. \end{proof}

Combinatorially, we can envision this Bernoulli decomposition as follows.  We imagine forming a permutation of $[n]$ by placing the elements $1, \ldots, n$ in cycles in turn.  When the element $k$ is inserted, with probability $\sigma/(\sigma+k-1)$ it is placed in a new cycle, and with probability $1/(\sigma+k-1)$ it is placed after any of $1, 2, \ldots, k-1$ in the cycle containing that element.  Then the probability of obtaining any permutation with $c$ cycles is $\sigma^c/(\sigma(\sigma+1)\cdots(\sigma+n-1))$, which is exactly the measure given to this permutation by $\P_\sigma^{(n)}$. This is an instance of the Chinese Restaurant Process \cite[Sec. 3.1]{Pi06}.

From this decomposition into Bernoulli random variables, we can derive formulas for the mean and variance of the number of cycles under the measure $\P_\sigma^{(n)}$.  In particular we note that since $X$ is a sum of Bernoulli random variables with small mean, the variance of $X$ is very close to its mean.  Let $\psi$ denote the digamma function $\psi(z) = \Gamma^\prime(z)/\Gamma(z)$; this has an asymptotic series $\psi(z) = \log z - {1 \over 2} z^{-1} - {1 \over 12} z^{-2} + O(z^{-4})$ as $z \to \infty$.  Let $H_n = \sum_{k=1}^n {1 \over k}$ be the $n$th harmonic number and let $\gamma = 0.57721\ldots$ be the Euler-Mascheroni constant.

\begin{prop}\label{ewens-exp-cycles} The expected number of cycles of a random $\sigma$-weighted permutation of $[n]$ is $\E_{\sigma}^{(n)} X = \sigma(\psi(n+\sigma) - \psi(\sigma))$; in particular if $\sigma$ is a positive integer we have \begin{equation}\label{ewens-exp-cycles-integer} \E_\sigma^{(n)} X = \sigma \log n + (\sigma \gamma - \sigma H_{\sigma-1}) + (\sigma^2 - \sigma/2) n^{-1} + O(n^{-2}). \end{equation}\end{prop}

\begin{proof} From Theorem \ref{ewens-bernoulli} we have 
\[ \E_\sigma^{(n)}X = \sum_{k=1}^n {\sigma \over \sigma+k-1} = \sigma \sum_{k=1}^n {1 \over \sigma+k-1}. \]
Now, $\psi(z+1) - \psi(z) = 1/z$; thus
\begin{eqnarray*}
\psi(n+\sigma) - \psi(\sigma) &=&  (\psi(n+\sigma) - \psi(n+\sigma-1)) + \cdots + (\psi(\sigma+1) - \psi(\sigma)) \\
&=& {1 \over n+\sigma-1} + {1 \over n+\sigma-2} + \cdots + {1 \over \sigma} \\
&=& \sum_{k=1}^n {1 \over \sigma+k-1}.
\end{eqnarray*}
This proves that $\E_\sigma^{(n)} X = \sigma(\psi(n+\sigma) - \psi(\sigma))$.  The asymptotic series follows from that for $\psi(z)$ where we have used the fact that $\psi(n) = H_{n-1}-\gamma$ when $n$ is a positive integer. 
\end{proof}

\begin{prop}\label{ewens-var-cycles} The variance of the number of cycles of a random $\sigma$-weighted permutation of $[n]$ is
\begin{equation}\label{eq:feb-6-beta} \sigma^2 \left( \psi^\prime(n+\sigma) - \psi^\prime(\sigma) \right) + \sigma(\psi(n+\sigma) - \psi(\sigma)); \end{equation}

this has an asymptotic series,
\begin{equation}\label{ewens-var-series} \V_\sigma^{(n)} X =  \sigma \log n + (-\sigma^2 \psi^\prime(\sigma) - \sigma \psi(\sigma)) + {4\sigma^2 -1 \over 2} n^{-1} +  O(n^{-2}) \end{equation}
\end{prop}
The proof is similar to that of the previous proposition, noting that the variance of a Bernoulli random variable with mean $p$ is $p-p^2$.

From (\ref{ewens-var-series}) we can also derive for integer $\sigma$ the explicit formula (not involving $\psi$)
\[ \V_\sigma^{(n)} X = -\sigma^2 \sum_{j=\sigma}^{\sigma+n-1} {1 \over  j^2}  + \sigma \left( \log n + \gamma - H_{\sigma-1} \right) + O(1/n) \]
which holds as $n \to \infty$.  It suffices to show that 
\begin{equation}\label{eq:feb-24-iota} \psi^\prime(n+\sigma) - \psi^\prime(\sigma) = -\sum_{j=\sigma}^{\sigma+n-1} {1 \over j^2}. \end{equation}
To see this, recall the identity $\psi(x+1)-\psi(x) = 1/x$; differentiating gives $\psi^\prime(x+1) - \psi^\prime(x) = -1/x^2$.  Summation over $x = \sigma, \sigma+1, \ldots, \sigma+n-1$ gives (\ref{eq:feb-24-iota}).

Finally, we recall a normal distribution result for the total number of cycles \cite[(5.22)]{ABT03}.  Let $\hat{X} = {X - \sigma \log n \over \sqrt{\sigma \log n}}$ be the standardization of $X$. Then $\lim_{n \to \infty} \P_\sigma^{(n)}(\hat{X} \le x) = \Phi(x)$, where $\Phi(x)$ is the cumulative distribution function of a standard normal random variable.  This follows from Theorem \ref{ewens-bernoulli} and the Lindeberg-Feller central limit theorem.

We have thus far looked at the total number of cycles of $\sigma$-weighted permutations.  These distributions, suitably scaled, are continuous in the large-$n$ limit. In contrast, looking at each cycle length separately, we approach a discrete distribution.  More specifically, the number of $k$-cycles of $\sigma$-weighted permutations of $[n]$, for large $n$, converges in distribution to $\Po(\sigma/k)$, where $\Po(\lambda)$ denotes a Poisson random variable with mean $\lambda$; here we consider how quickly $\E_\sigma^{(n)} X_k$ approaches $\sigma/k$.  Recall that $X_k$ is a random variable, with $X_k(\pi)$ the number of $k$-cycles of a permutation $\pi$.

\begin{prop}\label{prop-ewens-exp-k-cycles}\cite[(37)]{AT92}\cite{Wa74} The average number of $k$-cycles in a $\sigma$-weighted permutation of $[n]$  is  
\begin{equation}\label{ewens-exp-k-cycles} \E_\sigma^{(n)} X_k = {\sigma \over k} {(n)_k \over (n+\sigma-1)_k} \end{equation}
where $(n)_k = n(n-1)\ldots(n-k+1)$ is the ``falling power''. \end{prop}

We provide a new proof in terms of generating functions.

\begin{proof} The bivariate generating function counting $\sigma$-weighted permutations by their size and number of $k$-cycles is $P(z,u) = (1-z)^{-u\sigma} \exp(\sigma(u-1)z^k/k)$.  The mean number of $k$-cycles is given by
 \[ {[z^n] \left. \partial_u P(z,u) \right|_{u=1} \over [z^n] P(z,1)} = {[z^n] {\sigma z^k \over k} (1-z)^{-u\sigma} \over [z^n] (1-z)^{-\sigma}} = {\sigma \over k} {[z^{n-k}](1-z)^{-\sigma} \over [z^n] (1-z)^{-\sigma}} \]
and the binomial formula gives (\ref{ewens-exp-k-cycles}). \end{proof}

When $\sigma$ is an integer, a combinatorial proof can be obtained by considering $\sigma$-weighted permutations as permutations where each cycle is colored in one of $\sigma$ colors.  

\begin{prop}\label{ewens-asympt-k-cycles} There is an asymptotic series for $\E_\sigma^{(n)} X_k$,
 \[ \E_\sigma^{(n)} X_k = {\sigma \over k} \left( 1 - {(\sigma-1) k \over n} + O(n^{-2}) \right). \] \end{prop}
\begin{proof} The numerator and denominator of (\ref{ewens-exp-k-cycles}) are polynomials in $n$ of degree $k$; write the two highest-degree terms of each explicitly and divide. \end{proof}

\begin{prop}\label{ewens-asympt-alpha-n-cycles} Fix $0 < \alpha \le 1$.  The expected number of elements in $\alpha n$-cycles of a random $\sigma$-weighted permutation satisfies, as $n \to \infty$,
 \[ \E_\sigma^{(n)} Y_{\alpha n} = \sigma(1-\alpha)^{\sigma-1} + O(n^{-1}) \]
\end{prop}
(Here we have assumed for simplicity that $\alpha n$ is an integer.)
\begin{proof} Let $\beta = 1-\alpha$.  We have from Proposition \ref{prop-ewens-exp-k-cycles} that
 \[ \E_\sigma^{(n)} Y_{\alpha n} = \sigma {(n)_{\alpha n} \over (n+\sigma-1)_{\alpha n}} = \sigma {n!(\beta n +\sigma-1)! \over (\beta n)! (n+\sigma-1)!} = \sigma {n! \over (n+\sigma-1)!} {(\beta n + \sigma-1)! \over (\beta n)!} \]
We now note that $(n+r)!/n! = n^r (1+O(n^{-1}))$, for constant $r$ as $n \to \infty$, from Stirling's formula.    Applying this twice with $r = \sigma-1$ gives the result. \end{proof}

It would  be of interest to determine the limiting distribution of the number of cycles with length between $\gamma n$ and $\delta n$ for constants $\gamma$ and $\delta$. There can be at most $\lfloor \gamma^{-1} \rfloor$ such cycles, so this random variable is supported on $0, 1, \ldots, \lfloor \gamma^{-1} \rfloor$.  Thus to determine the limiting distribution it suffices to determine the 0th through $\lfloor \gamma^{-1} \rfloor$th moments of this random variable.  The $\sigma = 1$ case will be treated in \cite{Lu09}.

We can essentially integrate the result of Proposition \ref{ewens-asympt-k-cycles} to determine the number of elements in cycles with normalized length in a specified interval.   However, this can be done in a more general framework.  Following \cite[Thm. VI.1]{FS09} and \cite[Thm. 1]{FO90-transfer}, for constants $R>1$ and $\phi>0$ we define a $\Delta$-domain as a set of the form 
\[ \Delta(\phi, R) = \{ z : |z| < R, z \not = 1, |\arg(z-1)| > \phi \}. \] 
\begin{thm}\label{bulk-profile-general} Let $\sum_k \sigma_k z^k/k = \sigma \log {1 \over 1-z} + K + o(1)$ be analytic in its intersection with some $\Delta$-domain, for some constants $\sigma$ and $K$.  Then the probability that a uniformly chosen random element of a random $\vec{\sigma}$-weighted  permutation of $[n]$ lies in a cycle of length between $\gamma n$ and $\delta n$ approaches $(1-\gamma)^\sigma - (1-\delta)^\sigma$ as $n \to \infty$. \end{thm}

Note that analyticity in the slit plane suffices; this is the case $\phi = 0$.  We begin by stating two lemmas needed in the proof.

\begin{lemma}\label{bulk-profile-sum} Let $\{ \sigma_k \}_{k=1}^\infty$ be a sequence of nonnegative real numbers with mean $\sigma$.  Fix constants $0 \le \gamma < \delta < 1$. Then
\[ \lim_{n \to \infty} {1 \over n} \sum_{k = \lceil \gamma n \rceil}^{\lfloor \delta n \rfloor} \sigma_k \left( 1 - {k \over n} \right)^{\sigma-1} = (1-\gamma)^\sigma - (1-\delta)^\sigma. \] \end{lemma}

\begin{proof} We rewrite the sum as an integral,
\[ \sum_{k = \lceil \gamma n \rceil}^{\lfloor \delta n \rfloor} \sigma_k \left( 1 - {k \over n} \right)^{\sigma-1} = \int_{\gamma n}^{\delta n} \left( 1 - {k \over n } \right)^{\sigma-1} d \mu(k) \]
where $\mu(x) = \sum_{j=1}^{\lfloor x \rfloor} \sigma_j$.  Integrating by parts gives
\begin{equation}\label{bulk-profile-sum-1} (1-\delta)^{\sigma-1} \mu(\delta n) - (1-\gamma)^{\sigma-1} \mu(\gamma n) - \int_{\gamma n}^{\delta n} \mu(k) d \left( 1 - {k \over n} \right)^{\sigma-1}. \end{equation}
Differentiation allows us to rewrite the integral in (\ref{bulk-profile-sum-1}) as a Riemann integral,
\begin{equation}\label{bulk-profile-sum-2} \int_{\gamma n}^{\delta n} \mu(k) d \left( 1 - {k \over n} \right)^{\sigma-1} = {1-\sigma \over n} \int_{\gamma n}^{\delta n} \mu(k) \left( 1 - {k \over n} \right)^{\sigma-2} \: dk. \end{equation}
Let $\tau(k) = \mu(k) - \sigma k$.  Then the integral on the right-hand side of (\ref{bulk-profile-sum-2}) becomes
\begin{equation}\label{bulk-profile-sum-3} {1 - \sigma \over n} \left( \int_{\gamma n}^{\delta n} \sigma \left( 1 - {k \over n} \right)^{\sigma-2} \: dk + \int_{\gamma n}^{\delta n} \tau(k) \left( 1 - {k \over n} \right)^{\sigma-2} \right) \: dk. \end{equation}
We perform the first integral in (\ref{bulk-profile-sum-3}) and note that $\mu(\delta n) \sim \sigma \cdot \delta n, \mu(\gamma n) \sim \sigma \cdot \gamma n$ in (\ref{bulk-profile-sum-1}).  This gives
\begin{equation}\label{bulk-profile-sum-4} {1 \over n} \sum_k \left( 1 - {k \over n} \right)^{\sigma-1} \sim (1-\gamma)^\sigma - (1-\delta)^\sigma + {1-\sigma \over n^2} \int_{\gamma n}^{\delta n} \tau(k) \left( 1 - {k \over n} \right)^{\sigma-2} \: dk. \end{equation}
So it suffices to show that the final term in (\ref{bulk-profile-sum-4}) is negligible, i. e. 
\[ \int_{\gamma n}^{\delta n} \tau(k) \left( 1 - { k\over n} \right)^{\sigma-2} \: dk = o(n^2). \]
Since $\{ \sigma_k \}_{k=1}^\infty$ has mean $\sigma$, we have $\sum_{k=1}^n \sigma_k = \sigma n + o(n)$.  Thus $\tau(k) = o(n)$.  On $[\gamma n, \delta n]$, $(1-k/n)^{\sigma-2}$ is bounded.  So the integrand above is $o(n)$, and the integral is $o(n^2)$ as desired. \end{proof}

\begin{lemma}\label{bulk-profile-ratio} Say $[z^n] P(z) = Cn^{\sigma-1} (1 + o(1))$ uniformly in $n$, for some positive constants $C, \sigma$.  Then
\[ \sum_{k = \lceil \gamma n \rceil}^{\lfloor \delta n \rfloor} \sigma_k {[z^{n-k}] P(z) \over [z^n] P(z)} \sim \sum_{k = \lceil \gamma n \rceil}^{\lfloor \delta n \rfloor} \sigma_k \left( 1 - {k \over n} \right)^{\sigma-1} \]
as $n \to \infty$, for any $0 \le \gamma  < \delta < 1$. \end{lemma}
\begin{proof} From the hypothesis that $[z^n]P(z) \sim Cn^{\sigma-1}$, we get
\[ {[z^{n-k}] P(z) \over [z^n] P(z)} \sim {C(n-k)^{\sigma-1} \over Cn^{\sigma-1}} = \left( 1 - {k \over n} \right)^{\sigma-1} \]
uniformly as $n, k \to \infty$ with $0 \le k < \delta n$.  Therefore
\begin{eqnarray*}
\sum_{k = \lceil \gamma n \rceil}^{\lfloor \delta n \rfloor} \sigma_k {[z^{n-k}] P(z) \over [z^n] P(z)} &=& \sum_{k = \lceil \gamma n \rceil}^{\lfloor \delta n \rfloor} \sigma_k \left( 1 - {k \over n} \right)^{\sigma-1} (1 + o(1)) \\
&=& \sum_{k = \lceil \gamma n \rceil}^{\lfloor \delta n \rfloor} \sigma_k \left( 1 - {k \over  n} \right)^{\sigma-1} + \sum_{k = \lceil \gamma n \rceil}^{\lfloor \delta n \rfloor} \sigma_k \cdot o(1) \cdot \left( 1 - {k \over n} \right)^{\sigma-1}. \label{foo} 
\end{eqnarray*}
The first sum in the previous equation is $\Theta(n)$.  The second sum has $\Theta(n)$ terms; since $(1-k/n)^{\sigma-1}$ and $\sigma_k$ can both be bounded above on the interval $[\gamma n, \delta n]$ each term is $o(1)$.  Thus the second sum is $o(n)$.  So
\begin{eqnarray*}
\sum_{k = \lceil \gamma n \rceil}^{\lfloor \delta n \rfloor} \sigma_k {[z^{n-k}] P(z) \over [z^n] P(z)} &=& \sum_{k = \lceil \gamma n \rceil}^{\lfloor \delta n \rfloor} \left( \sigma_k \left( 1 - {k \over n} \right)^{\sigma-1} \right) + o(n) \\
&=& \sum_{k = \lceil \gamma n \rceil}^{\lfloor \delta n \rfloor} \left( \sigma_k \left( 1 - {k \over n} \right)^{\sigma-1} \right) (1 + o(1)) \end{eqnarray*}
as desired. \end{proof}

\begin{proof}[Proof of Theorem \ref{bulk-profile-general}.] This probability can be written as
\[ \lim_{n \to \infty} {1 \over n} \sum_{k = \lceil \gamma n \rceil}^{\lfloor \delta n \rfloor} \sigma_k {[z^{n-k}] P(z) \over [z^n] P(z)}. \]
Now, recall 
\[ \sum_k \sigma_k z^k/k = \sigma \log {1 \over 1-z} + K + o(1) \]
by hypothesis.  Thus the generating function $P(z)$ of $\vec{\sigma}$-weighted permutations is
\[ P(z) = \exp \left( \sum_k \sigma_k z^k/k \right) = \exp \left( \sigma \log {1 \over 1-z} + K + o(1) \right) = (1-z)^{-\sigma} e^K (1+o(1)). \]
Applying the Flajolet-Odlyzko transfer theorem \cite{FO90-transfer}, $[z^n] P(z) = Cn^{\sigma-1} (1+o(1))$ for some positive real constant $C$.  Thus $P(z)$ satisfies the hypotheses of Lemma \ref{bulk-profile-ratio}.  Applying that lemma, we see that this sum is asymptotic to $n^{-1} \sum_{k = \lceil \gamma n \rceil}^{\lfloor \delta n \rfloor} \sigma_k (1-k/n)^{\sigma-1}$; the desired result then follows from Lemma \ref{bulk-profile-sum}. \end{proof}

The hypotheses, and hence the conclusions, of Theorem \ref{bulk-profile-general} hold for many weight sequences $\sigma_1, \sigma_2, \ldots$ with $\lim_{n \to \infty} {1 \over n} \sum_{k=1}^n \sigma_k = \sigma$; that is, for weight sequences averaging $\sigma$.  In particular, we have the following special case.

\begin{cor}\label{bulk-profile-ewens} Fix constants $0 \le \gamma \le \delta \le 1$.  Let $p_\sigma(n;\gamma,\delta)$ be the probability that the element $1$, in a $\sigma$-weighted permutation of $[n]$, lies in a cycle of length in the interval $[\gamma n, \delta n]$.  Then
 \[ \lim_{n \to \infty} p_\sigma(n; \gamma,\delta) = (1-\gamma)^\sigma - (1-\delta)^\sigma. \]
\end{cor}

\begin{proof} We have the cycle generating function $\sum_{k=1}^\infty \sigma z^k/k = \sigma \log {1 \over 1-z}$; apply Theorem \ref{bulk-profile-general}. \end{proof}

For example, setting $\sigma = 1/2, \gamma = 0.99, \delta = 1$, we see that for large $n$, 10\% of elements of $1/2$-weighted permutations are in cycles of length at least $0.99n$.  If we define the ``co-length'' of a cycle of a permutation to be the number of elements {\it not} in that cycle, a cleaner statement of the theorem becomes possible.  The proportion of elements of $\sigma$-weighted permutations in cycles of co-length at most $\zeta n$ is $\zeta^\sigma$.

It would be desirable to replace the condition in the hypothesis of Theorem \ref{bulk-profile-general} with the less restrictive 
\[ \sum_k {\sigma_k z^k \over k} = \sigma \log {1 \over 1-z} \cdot (1 + o(1)); \]
it seems  likely that this suffices to prove a limit law but the proof does not easily adapt to that case.

\section{Permutations with all cycle length of the same parity} 

This section is devoted to results on random permutations in which all cycle lengths have the same parity; that is, they are either all even or all odd.  We adopt the notation $\P_e^{(n)}$ for the family of measures $\P_{\vec{\sigma}}^{(n)}$ where $\vec{\sigma} = (0, 1, 0, 1, \ldots)$, and similarly $\P_o^{(n)}$ for the family with $\vec{\sigma} = (1, 0, 1, 0, \ldots)$; these are the measures corresponding to permutations with all cycle lengths even and with all cycle lengths odd, respectively.  

The results obtained here resemble those for the Ewens sampling formula with parameter $1/2$.  A heuristic explanation for this phenomenon is as follows.  Let us produce a permutation of $[n]$ from the Ewens distribution with parameter $1/2$ by first picking a permutation $\pi$ uniformly at random from $S_n$, and then flipping a fair coin for each cycle of  $\pi$.  If all the coins come up heads we keep $\pi$; otherwise we ``throw back'' the permutation $\pi$ and repeat this process until we have a trial in which all coins come up heads.  The number and normalized size of cycles of permutations obtained in this manner should be similar to those of permutations with all cycle lengths even, since for large permutations the parity constraint is essentially equivalent to a coin flip.

\begin{prop}\label{exp-k-cycles-even} The expected number of elements in $k$-cycles of a permutation of $[n]$ with all cycle lengths even is
\[ \E_e^{(n)} Y_k = {n(n-2) \cdots (n-k+2) \over (n-1)(n-3) \cdots (n-k+1)} \]
if $k$ is even, and 0 if $k$ is odd. \end{prop}

\begin{proof} By Lemma \ref{fundamental-lemma}, we have $\E_e^{(n)} Y_k = {[z^{n-k}] (1-z^2)^{-1/2} / [z^n] (1-z^2)^{-1/2}}$; we apply the binomial theorem and simplify.  \end{proof} 

For example, when $n = 10$ we have
\begin{eqnarray*} \left( \E_e^{(10)}Y_2, \E_e^{(10)}Y_4, \ldots, \E_e^{(10)}Y_{10} \right) &=& (10/9, 80/63, 32/21, 128/63, 256/63) \\
 &\approx& (1.11, 1.27, 1.52, 2.03, 4.06) \end{eqnarray*}
and we observe that most entries are in the longer cycles.  For $n = 100$ this is illustrated in the figure above.
\begin{figure}\label{figure-even-100} \begin{center} \includegraphics[width=0.5\textwidth]{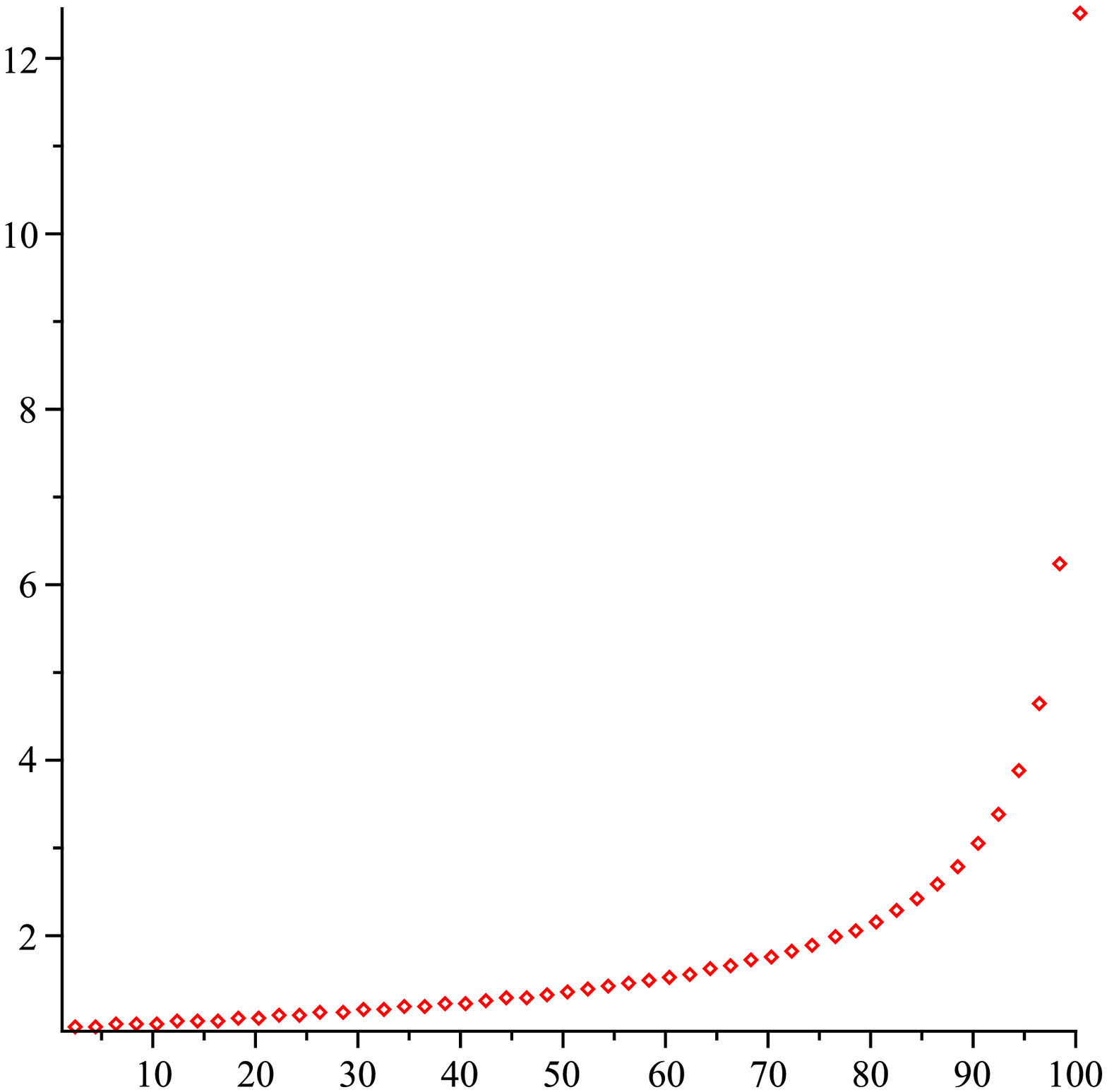} \end{center} \caption{$\E_e^{(100)} Y_k$ for $k = 2, 4, \ldots, 100$}. \end{figure}


\begin{prop}\label{exp-k-cycles-odd} The expected number of elements in $k$-cycles of a random permutation of $[n]$ with all cycle lengths odd is
\[ \E_o^{(n)} Y_k = \begin{cases}
{n(n-2) \cdots (n-k+1) \over (n-1)(n-3)\cdots(n-k)}, & \text{$n$ even} \\
{(n-1)(n-3) \cdots (n-k+2) \over (n-2)(n-4) \cdots (n-k+1)}, & \text{$n$ odd}  \\
\end{cases} \] \end{prop}

\begin{proof}   The generating function of permutations with all cycle lengths odd, counted by their number of cycles, is $P(z,u) = ((1+z)/(1-z))^{u/2}$.  We use Lemma \ref{fundamental-lemma} to see that the mean number of elements in $k$-cycles is given by $\E_o^{(n)} Y_k = {[z^{n-k}] \sqrt{1+z \over 1-z} / [z^n] \sqrt{1+z \over 1-z}}.$  We recall that $[z^n] \sqrt{1+z \over 1-z}$ is ${(n-1)!!^2 \over n!}$ if $n$ is even and ${n!!(n-2)!! \over n!}$ if $n$ is odd; substituting and simplifying gives the result.
\end{proof}

By similar methods, we can obtain formulas for the exact number of permutations of $[n]$ with all cycle lengths divisible by $a$, and the exact expected number of $k$-cycles in such permutations for integers $k$ which are divisible by $a$. These permutations have exponential generating function $(1-z^a)^{-1/a}$.  Permutations with all cycle lengths congruent to $k$ mod $a$ for some nonzero $k$ are more difficult to deal with, as it appears that the generating function cannot be written in an elementary form except when $a$ is even and $k = a/2$.  (See \cite[Sec. 5.0.3]{Sa97} for the relevant generating functions.)

\begin{prop}\label{asymp-exp-k-cycles-odd}
\renewcommand{\labelenumi}{(\alph{enumi})} 
\begin{enumerate} 
 \item The number of elements of $k$-cycles in a permutation of $[n]$ with all cycle lengths odd, for fixed odd $k$, satisfies $\E_o^{(n)} Y_k = 1 + {k+1 \over 2n} + O(n^{-2})$ as $n$ approaches $\infty$ through even values, and $\E_o^{(n)} Y_k = 1 + {k-1 \over 2n} + O(n^{-2})$ as $n$ approaches $\infty$ through odd values. 
\item The number of elements of $k$-cycles in a permutation of $[n]$ with all cycle lengths even, for fixed even $k$, satisfies $\E_e^{(n)} Y_k = 1 + {k \over 2n} + O(n^{-2})$ as $n$ approaches infinity through even values.
\end{enumerate}
\end{prop}
\begin{proof} To prove (a), from Proposition \ref{exp-k-cycles-odd} we have previous formulas for $\E_o^{(n)} Y_k$ depending on the parity of $n$.  These are fractions which have numerators and denominators which are polynomials in $n$; we can write out the two highest-degree terms of each polynomial and simplify.  To prove (b) we proceed similarly from Proposition \ref{exp-k-cycles-even}. \end{proof}  

Note that the expected number of elements in $k$-cycles of permutations with all cycle lengths even (or odd) approaches 1 as $n$ gets large, if $k$ has the appropriate parity.   Assume we are dealing with permutations with all cycle lengths even.  Naively, we might add the limits of the expected number of elements in $2$-cycles, $4$-cycles, \ldots, $n$-cycles, and expect to get $n$.  But these are each 1; their sum is $n/2$.  Since each element is in a cycle, we must have $\sum_{k=1}^n \E_e^{(n)} Y_k = n$.  The difficulty is that the convergence of $\E_e^{(n)} Y_k$ as $n \to \infty$ is not uniform over $k$.   Under the correct scaling, then, subtler phenomena can be seen; the ``missing'' elements end up disproportionately in the longer cycle lengths for permutations with all cycle lengths even.  We note that similar phenomena of nonuniform convergence have previously been observed in random mappings, for example in \cite{FO90-mapping}. 

\begin{prop}\label{exp-alpha-n-cycles}  Fix $\epsilon \in (0,1)$.  The expected number of elements in $k$-cycles in a random permutation of $[n]$ with all cycle lengths even satisfies {\it uniformly} 
\[ \E_e^{(n)} Y_k \to \left( 1 - {k \over n} \right)^{1/2} \]
as $k, n \to \infty$ with $0 < k/n < 1-\epsilon$.
\end{prop}

\begin{proof}  The result of Proposition \ref{exp-k-cycles-even} can be rewritten in terms of factorials as
\[ \E_e^{(n)} Y_k = 2^k \left( {(n/2)! \over ((n-k)/2)!} \right)^2 {(n-k)! \over n!} \]
and by Stirling's approximation and routine simplifications, we have
\begin{equation}\label{eq:stirling-consequence} \E_e^{(n)} Y_k = \left( 1 - {k \over n} \right)^{-1/2} { 1 + {1 \over 4n} + O(n^{-2}) \over 1  + {1 \over 4(n-k)} + O((n-k)^{-2}) }. \end{equation}
Let $n, k \to \infty$ with $0 < k/n < 1-\epsilon$.  Then we have $1/(4(n-k)) \in [(4n)^{-1}, (4\epsilon n)^{-1})]$, and so $O((n-k)^{-2}) = O(n^{-2})$.  Furthermore $1/(4(n-k)) = O(n^{-1})$, with the constant implicit in the $O$-notation being $(4\epsilon)^{-1}$.  Therefore
\[ \E_e^{(n)} Y_k = \left( 1 - {k \over n} \right)^{-1/2} (1 + O(n^{-1})) \]
uniformly, as $k, n \to \infty$ with $0 < k/n < 1-\epsilon$.
\end{proof}

Furthermore, we can essentially integrate the result of Proposition \ref{exp-alpha-n-cycles} to determine the cumulative distribution function of the length of the cycle containing a random element of a random permutation with all cycle lengths even (or odd).  This is the content of the next theorem.    
 
\begin{thm}\label{bulk-profile-even} 
 Fix constants $0 \le \gamma \le \delta \le 1$.  Let $p_e(n;\gamma,\delta)$ be the probability that $1$ is contained in a cycle of length between $\gamma n$ and $\delta n$ of a permutation chosen uniformly at random from all permutations of $[n]$ with all cycle lengths even.  Then
\[ \lim_{n \to \infty} p_e(n; \gamma, \delta) = \sqrt{1-\gamma} - \sqrt{1-\delta}. \]
\end{thm}

Since the measure $\P_e$ is invariant under conjugation, this is the probability that an element of $[n]$ chosen uniformly at random is in a cycle of length between $\gamma n$ and $\delta n$ in a random permutation of $[n]$ with all cycle lengths even.
\begin{proof} Note that
\[ \sum_{2|k} {z^k \over k} = {1 \over 2} \log {1+z \over  1-z} = {1 \over 2} \log {1 \over 1-z} + \log 2 + o(1) \]
and apply Theorem \ref{bulk-profile-general}.
 \end{proof}

The same is true for permutations with all cycle lengths odd; like those with all cycle lengths even they fall in the ``$\sigma = 1/2$ class''.

We now move to consider the mean and variance of the total number of cycles of all lengths.
 
\begin{thm}\label{thm-mean-cycles-odd}
The mean number of cycles of a randomly chosen permutation of $[n]$ with all cycle lengths odd is, as $n \to \infty$,
\[ {1 \over 2} \log  n + {\gamma + 3 \log 2 \over 2} \pm {\gamma + \log n \over 8n} + O \left( {\log n \over n^2} \right) \]
where we take the $+$ sign if $n$ is odd and the $-$ sign if $n$ is even.  The variance of the number of cycles is, as $n \to \infty$, 
\[ {1 \over 2} \log n + {\gamma + 3 \log 2 - 4\pi^2 \over 8} + O \left( {\log^2 n  \over n} \right). \]
\end{thm}

\begin{proof}
We have the exponential generating function counting such permutations by size and number of cycles, $\left({1 + z \over 1-z}\right)^{u/2}$.  We can differentiate to obtain the mean and variance of the number of cycles. These are given by
\[ \mu_n := {[z^n] {1 \over 2} \sqrt{1+z \over 1-z} \log {1+z \over 1-z} \over [z^n] \sqrt{1+z \over 1-z}}, \sigma_n^2 := {[z^n] {1 \over 4} \sqrt{1+z \over 1-z} \log^2 {1+z \over 1-z} \over [z^n] \sqrt{1+z \over 1-z}} + \mu_n - \mu_n^2. \]
Let $f_r(z) = \sqrt{1+z \over 1-z} \log^r {1+z \over 1-z}$ for $r = 0, 1, 2$ and let $a_r(n) = [z^n] f_r(z)$ for $r = 0, 1, 2$.  Then we have 
\begin{equation}\label{eq:feb-24-lambda} \mu_n = {a_1(n) \over 2a_0(n)}, \sigma_n^2 = {a_2(n) \over 4a_0(n)} + \mu_n - \mu_n^2\end{equation}
and we need to find asymptotic series for the $a_r(n)$ as $n \to \infty$.  We observe that $a_0(n)$ is the number of permutations of $[n]$ with all cycle lengths odd, which is $(n-1)!!^2/n!$ if $n$ is even and $n!!(n-2)!!/n!$ if $n$ is odd; Stirling's formula gives an asymptotic expansion, depending on the parity of $n$.  To find a series for $a_1(n)$ as $n \to \infty$, we expand $f_1(z) = \sqrt{1+z \over 1-z}$ in a series with terms which are half-integral powers of $1-z$.  From this we derive a series for $\sqrt{1+z \over 1-z} \log {1+z \over 1-z}$ with terms of the form  $(1-z)^{i-1/2} L^j$ where $L = \log 1/(1-z)$.  The function being expanded is analytic in the complex plane slit along the real half-line $\{ z \in \mathbb{R} : z \ge 1\}$; by \cite[Thm. VI.3]{FS09}, an error of $O((1-z)^{i-1/2} L)$ in the series for $f_1(z)$ leads to an error $O(n^{-i-1/2} \log n)$ in the series for $a_1(n)$.  We can thus transfer an asymptotic expansion for $f_1(z)$ near $z=1$ to give an expansion for $a_1(n)$ as $n \to \infty$, and similarly for $f_2(z)$ and $a_2(n)$.  Combining these series as specified by (\ref{eq:feb-24-lambda}) gives the result. \end{proof}

We observe that this is $(\log 2) + o(1)$ more than the number of cycles of a permutation of $[n]$ with all cycle lengths even.

The following two results give a decomposition of the number of cycles of permutations with all cycle lengths even into a sum of Bernoulli random variables.

\begin{thm}\label{even-gf} The generating function of permutations of $[2n]$ with all cycle lengths even, counted by their number of cycles, is
\begin{equation}\label{eq:feb-24-epsilon} p_{2n}(u) = [u(u+2)(u+4) \cdots (u+(2n-2))] \cdot (2n-1)!! \end{equation}
\end{thm}

\begin{proof} The bivariate generating function for permutations with all cycle lengths even, counted by their size and number of cycles, is $(1-z^2)^{-u/2}$.   Let $p_k(u)$ be the desired generating function.  Then we have
 \[ (1-z^2)^{-u/2} = p_0(u) + p_1(u) z + p_2(u) {z^2 \over 2!} + \cdots \]
and it is clear that $p_k$ is the zero polynomial for odd $k$.  For even $k$, the binomial theorem gives
\[ (1-z^2)^{-u/2} = 1 + {-u/2 \choose 1} (-z^2) + {-u/2 \choose 2} (-z^2)^2 + \cdots \]
and so we have $p_{2n}(u) = (2n)! {-u/2 \choose n}$ by comparing coefficients; this can be expanded to give the expression above.
\end{proof}

A combinatorial proof is also possible.  Recall that we can write a permutation $\pi$ of $[n]$ in terms of its {\it inversion table}, a sequence of integers $a_1, a_2, \ldots, a_n$, with $a_i = |\{ j : j < i, \pi(j) > \pi(i)\}|$.  The number of zeros in the sequence $(a_1, \ldots, a_n)$ is the number of left-to-right maxima of $\pi$.  The ``fundamental correspondence'' between permutations written in cycle notation and in one-line notation takes permutations with $k$ cycles to those with $k$ left-to-right maxima; furthermore, permutations with all cycle lengths even are taken to those with all left-to-right maxima in odd positions, and conversely.  Thus it suffices to show that $p_n(u)$ is the generating function of permutations with all left-to-right maxima in odd positions, counted by their number of maxima; this is done by considering the inversion table.

\begin{thm}\label{thm-mean-cycles-even} The number of cycles $C_n$ of a random permutation of $[2n]$ with all cycle lengths even, as $n \to \infty$, is asymptotically normally distributed with
\begin{equation}\label{eq:feb-24-zeta} \E(C_n) = {1 \over 2} \log n + \left( {1 \over 2} \gamma + \log 2 \right) + O(n^{-1}) \end{equation}
 and
\begin{equation}\label{eq:feb-24-eta}  \V(C_n) = {1 \over 2} \log n + \left( {1 \over 2} \gamma + \log 2 - {\pi^2 \over 8} \right) + O(n^{-1}) \end{equation} \end{thm}

\begin{proof} Let $n = 2m$.  From Theorem \ref{even-gf}, we have $C_m = \sum_{k=1}^m X_{m,k}$ where the $X_{m,k}$ are independent Bernoulli random variables with $\P(X_{m,k} = 1) = 1/(2k-1)$.  The formula (\ref{eq:feb-24-zeta}) for the expectation follows from the asymptotic series for the harmonic numbers.  The variance is given by
\[ \V C_m = \sum_{k=1}^m \left( {1 \over 2k-1} - \left( {1 \over 2k-1} \right)^2 \right) = \E C_m - \sum_{k=1}^m {1 \over (2k-1)^2}. \]
and we need to consider the second sum.  We have $\sum_{j=1}^m {1 \over j^2} = -\psi^\prime(m+1) + \pi^2/6$, so 
\[ \V C_m = \E C_m - \left( -\psi^\prime(2n+1) + {1 \over 4} \psi^\prime(n+1) + {\pi^2 \over 8} \right).\]
But $\psi^\prime(m) = O(m^{-1})$, so in fact we get $\V C_m = \E C_m - {\pi^2 \over 8} + O(1/m)$; from this and (\ref{eq:feb-24-zeta}) we get (\ref{eq:feb-24-eta}).  Asymptotic normality follows from the Lindeberg-Feller central limit theorem \cite{Du}. \end{proof}
 
There is not such a simple decomposition for permutations with all cycle lengths odd.  However, it appears that the polynomials counting permutations of $[n]$ with all cycle lengths odd by their number of cycles have only pure imaginary roots.  If this is true, then the number of cycles of a random permutation of $[n]$ with all cycle lengths odd can be decomposed into a sum of $\lfloor n/2 \rfloor$ independent $\{0, 2\}$-valued random variables, plus 1 if $n$ is odd.    It may be of interest to study the zeros of these polynomials.

\section{Boltzmann sampling}

We have at this point seen substantial similarities between permutations with all cycle lengths having the same parity and permutations with cycle weights $1/2$.  This suggests that an average of weights is in some sense a more fundamental parameter than the individual weights.  This has been anticipated by the notion of a {\it function of logarithmic type} \cite{FS90}, which has been used in the study of permutations \cite{Ha94}.  Let $\Delta_0(\rho, \eta) = \{ z: |z| < \rho + \eta, z \not \in [\rho, \rho + \eta] \}$.  A function $G(z)$ is called {\it logarithmic} if it is of the form
\[ G(z) = a \log {1 \over 1-z/\rho} + R(z) \]
for some constant multiplier $a$ and function $R(z)$, where $R(z)$ is analytic in $\Delta_0$ and satisfies $R(z) = K + o(1)$ for some constant $K$ as $z \to \rho$ in $\Delta_0$, and $\rho$ is the unique dominant singularity of $G$ on its circle of convergence.   
In \cite[Prop. 1]{FS90} structures having components enumerated by a function of logarithmic type $G(z)$ are considered; for such structures of size $n$, the expected number of cycles is $a \log n + O(1)$, as is the variance.  However, the structures considered in this paper have not all had components counted by functions of logarithmic type.  For example, the components of permutations with all cycle lengths even are counted by the exponential generating function ${1 \over 2} \log {1 \over 1-z^2}$, which has singularities at $z = \pm 1$ and thus does not have a unique dominant singularity.

The following conjecture, in the light of these averaging phenomena, seems natural.  It is supported by Theorem \ref{thm-boltzmann-concentrated}, an analogous result on ``Boltzmannized'' permutations.

\begin{conj}\label{normal-conjecture} Let $\vec{\sigma} = (\sigma_1, \sigma_2, \ldots)$ be a sequence of nonnegative real numbers with mean $\alpha$, that is, with $\lim_{n \to \infty} {1 \over n} \left( \sum_{k=1}^n \sigma_k \right) = \alpha$.  Then permutations of $[n]$ selected according to the weights $\vec{\sigma}$ have an asymptotically Gaussian number of cycles as $n \to \infty$, with mean and variance asymptotic to $\alpha \log n$. \end{conj}

\begin{defin} Let $\vec{\sigma} = (\sigma_1, \sigma_2, \ldots)$ be a weighting sequence, and let $x$ be a positive real parameter.  Let $|\pi|$ denote the size of the ground set of a permutation $\pi$. Then we define the {\it $\vec{\sigma}$-weighted Boltzmann measure with parameter $x$} on permutations, a probability measure on $\bigcup_{k=0}^\infty S_k$, by
 \[ \P_{\vec{\sigma},x}(\pi) = {w_{\vec{\sigma}}(\pi) \cdot {x^{|\pi|} \over |\pi|!} \over \exp \left( \sum_{k \ge 1} \sigma_k x^k / k \right)} \]
\end{defin}  (See Definition \ref{def:weighting-sequence} for the weight $w_{\vec{\sigma}}(\pi)$.) 

For any choice of $\vec{\sigma}$ and $x$, $\P_{\vec{\sigma},x}$ is a probability measure.  It suffices to show that $\P_{\vec{\sigma},x}$ has total mass 1.  But $\sum_{k \ge 1} \sigma_k x^k/k$ is the weighted generating function of cycles, and we can apply the exponential formula.  Thus $\P_{\vec{\sigma},x}$ is a straightforward weighted generalization of the Boltzmann measure on labelled objects studied in \cite{DFLS04, FFGPP06}.  We also retain the formulas from \cite[Thm 2.1]{DFLS04} for the expected size and the variance of the size of the objects chosen according to this measure.  Let $C(x)$ be the exponential generating function of a labelled combinatorial class, and $N$ the size of a random object chosen from that class according to the Boltzmann measure.  Then 
\[ \E_{\vec{\sigma},x}(N) = {x {d \over dx} C(x) \over C(x)}, \E_{\vec{\sigma},x}(N^2) = {\left( x {d \over dx} \right)^2 C(x) \over C(x)}. \] 

We now assemble a sequence of lemmas.  These lemmas will be used to prove the following theorem, which is the main result of this section.

\begin{thm}\label{thm-boltzmann-concentrated} Let $\vec{\sigma} = (\sigma_1, \sigma_2, \ldots)$ be a weighting sequence with mean $\alpha$.  Let $x = x(\mu)$ be chosen so that $\E_{\vec{\sigma},x}(N) = \mu$.  Let $X$ be a random variable denoting the number of cycles of a permutation.  Then $\E_{\vec{\sigma},x} (X) = \V_{\vec{\sigma},x}(X) \sim \alpha \log \mu$ as $x \to 1^-$ or $\mu \to \infty$.
\end{thm}

The main analytic result needed follows. 
\begin{lemma}\label{lemma-polya-szego} \cite[Exercise I.88]{PS98}   Let $b_0, b_1, \ldots$ be positive real numbers, such that $\sum_{n=0}^\infty b_n$ is divergent, and $\sum_{k \ge 0} b_k t^k$ is convergent for $0 \le t < 1$. Then
\[ \lim_{n \to \infty} {a_0 + a_1 + \cdots + a_n \over b_0 + b_1 + \cdots + b_n} = s \textrm{ implies } \lim_{t \to 1^-} {\sum_{k \ge 0} a_k t^k \over \sum_{k \ge 0} b_k t^k} = s. \] \end{lemma}

\begin{lemma}\label{lemma-boltzmann-one}
Let $\sigma_1, \sigma_2, \sigma_3, \ldots$ be a sequence of real numbers such that 
\[ \lim_{n \to \infty} {1 \over n} \sum_{k=1}^n \sigma_k = \alpha. \]
Then $ \sum_{k=1}^\infty \sigma_k x^k = {\alpha \over 1-x} + o((1-x)^{-1}). $ \end{lemma}
\begin{proof} Apply Lemma \ref{lemma-polya-szego} with $a_k = \sigma_k, b_k = 1$. \end{proof}

\begin{lemma}\label{lemma-boltzmann-two} Let $\{\sigma_k\}_{k=1}^\infty$ be a sequence of nonnegative real numbers, bounded above, such that $\sum_{k=1}^n \sigma_k \sim \alpha n$ as $n \to \infty$, for some constant $\alpha >0$.  Then $\sum_{k=1}^n {\sigma_k \over k} \sim \alpha \log n$ as $n \to \infty$. \end{lemma}

\begin{proof}  We begin by showing that if $\int_1^n f(x) \: dx \sim n$ as $n \to \infty$ for some function $f$ such that $f(x)/x$ is integrable on $[1, \infty)$, then $\int_1^n {f(x) \over x} \:dx \sim \log n$ as $n \to \infty$.  Let $F(x) = \int_1^n f(x) \: dx$.  We integrate $\int_1^n {f(x) \over x} \: dx$ by parts, getting
 \[ \int_1^n {f(x) \over x} \: dx = {F(n) \over n} - {F(1) \over 1} + \int_1^n {F(x) \over x^2} \: dx. \]
Clearly $F(1) = 0$, and $F(n) \sim n$ by assumption, so 
\[ \int_1^n {f(x) \over x} \: dx = 1 + \int_1^n {F(x) \over x^2} \: dx  + o(1) \]
 Since $F(x) \sim x$ as $x \to \infty$, the integrand satisfies $F(x)/x^2 \sim 1/x$, and so
\[ \int_1^n {F(x) \over x^2} \: dx \sim \int_1^n {1 \over x} \: dx = \log n, \]
proving the claim.

Now, we need to check that this statement about integrals translates into an analogous one about sums.  Let $\{ \sigma_k \}_{k=1}^\infty$ be as in the hypothesis, and let $f(x) = \sigma_{\lfloor x \rfloor}$.  Then we want to show that
\[ \int_1^{n+1} {f(x) \over x} \: dx - \sum_{k=1}^n {\sigma_k \over k} = o(\log n) \]
as $n \to \infty$.    We have
\begin{eqnarray*}
 \int_1^{n+1} {f(x) \over x} \: dx - \sum_{k=1}^n {f(k) \over k} &=& \sum_{k=1}^n \left( \int_k^{k+1} {f(x) \over x} \: dx - {f(k) \over k} \right) \\
&=& \sum_{k=1}^n f(k) \left( \log \left( 1 + {1 \over k} \right) - {1 \over k} \right) 
\end{eqnarray*}
and so, since $|\log(1 + 1/k) - 1/k|  \le 1/2k^2$ for positive integer $k$,
\[
 \left| \int_1^{n+1} {f(x) \over x} \: dx - \sum_{k=1}^n {f(k) \over k} \right| \le \left| \sum_{k=1}^n {f(k) \over 2k^2} \right|.
\]
Since $\{ f(k) \}_{k=1}^\infty$ is bounded, the sum on the right-hand side is convergent.  We have $\int_1^{n+1} f(x)/x \: dx \sim \log n$, so $\sum_{k=1}^n f(k)/k \sim \log n$ as well.  Thus we have proven the lemma for $\alpha =1$.  Multiplying through by $\alpha$ gives the desired result. \end{proof}

\begin{lemma}\label{lemma-boltzmann-three} Let $\sigma_1, \sigma_2, \sigma_3, \ldots$ be a sequence of positive real numbers such that \linebreak
$\lim_{n \to \infty} {1 \over n} \sum_{k=1}^n \sigma_k = \alpha.$ Then 
\begin{equation}\label{eq:feb-24-theta} \sum_{k=1}^\infty {\sigma_k x^k \over k} = \alpha \log {1 \over 1-x} + o
\left( \log {1 \over 1-x} \right) \end{equation} \end{lemma}

\begin{proof}  Applying Lemma \ref{lemma-boltzmann-two} to the hypothesis, $\lim_{n \to \infty} {1 \over \log n} \sum_{k=1}^n {\sigma_k \over k} = \alpha$. 
We apply Lemma \ref{lemma-polya-szego} with $a_k = \sigma_k/k, b_k = 1/k$.  This gives us
\[ \lim_{n \to \infty} {\sum_{k=0}^n \sigma_k / k \over 1 + H_n} = \lim_{x \to 1^-} {\sum_{k \ge 1} \sigma_k x^k / k \over \sum_{k \ge 1} x^k / k}. \]
Now, $\sum_{k \ge 1} x^k / k = \log (1/(1-x))$, and $1 + H_n \sim \log n$, so we have
\[ \lim_{n \to \infty} {1 \over \log n} \sum_{k=0}^n {\sigma_k \over k} = \lim_{x \to 1^-} { {\sum_{k \ge 1} \sigma_k x^k/k} \over \log (1/(1-x))}. \]
Thus the right-hand side here has value $\alpha$, proving (\ref{eq:feb-24-theta}).
\end{proof}

\begin{proof}[Proof of Theorem \ref{thm-boltzmann-concentrated}]  
Note that for the Boltzmann measure with parameter $x$ and weight sequence $\vec{\sigma}$, we have $C(x) = \exp \left( \sum_{k \ge 1} \sigma_k x^k/k \right)$.  Thus the distribution of sizes $N$ under this measure has expectation
\[  \E_{\vec{\sigma}, x} (N) = x {d \over dx} \left( \sum_{k \ge 1} {\sigma_k x^k \over k} \right) = \sum_{k \ge 1} \sigma_k x^k \]
Furthermore, the Boltzmann distribution $\P_{\vec{\sigma},x}$ can be obtained by taking $\Po(\sigma_k x^k/k)$ cycles of length $k$ for each $k \ge 1$ and forming uniformly at random a permutation with the resulting cycle type. The mean and variance of the number of cycles chosen from the distribution $\P_{\vec{\sigma},x}$ is thus exactly $\sum_{k \ge 1} \sigma_k x^k/k$.   Since we have $\E_x(N) \sim \alpha/(1-x)$ as $x \to 1^-$ by Lemma \ref{lemma-boltzmann-one}, we can solve for $x$ to see that $ 1 - {\alpha \over \E_x(N)} \sim x$ as $x \to 1^-$. Therefore 
\[  \sum_{k \ge 1} \sigma_k x^k/k \sim \alpha \log {1 \over 1-x} \sim \alpha \log {1 \over 1 - \left( 1 - {\alpha \over \E_x(N)} \right)} 
= \alpha \log { \E_x(N) \over \alpha } \sim \alpha \log \E_x(N) \]
which is the desired result.
\end{proof}

It would be desirable to translate Theorem \ref{thm-boltzmann-concentrated} into a result about permutations of a fixed size selected uniformly at random; this is one possible way of proving Conjecture \ref{normal-conjecture}.  Note that $\P_{\vec{\sigma},x}$ is a mixture of the various $\P_{\vec{\sigma}}^{(n)}$.  It is often possible to prove results about a family of measures $P_\lambda$, parametrized by $\lambda$, which are mixtures of well-understood measures $\P^{(n)}$, where we draw from $\P^{(n)}$ with probability $e^{-\lambda} \lambda^n/n!$; this goes by the name of analytic de-Poissonization \cite{JS98}, \cite[Ch. 10]{Sz01}.  Informally, we pick from $\P^{(N)}$ where $N$ is Poisson with parameter $\lambda$.  In the case described here we can get results on permutations chosen from $\P^{(N)}$ where $N$ is the size of objects from a Boltzmann distribution; thus techniques of ``analytic de-Boltzmannization'' will be necessary to achieve this goal.

{\it Acknowledgement.}  I thank Robin Pemantle, my Ph.D. thesis advisor, for very helpful suggestions and valuable support.  I thank the anonymous referee for remarks concerning the proof of Theorem \ref{bulk-profile-general}.

\end{document}